\documentclass{mcom-l}

\newtheorem{theorem}{Theorem}[section]
\newtheorem{lemma}[theorem]{Lemma}

\theoremstyle{definition}

\newtheorem{corollary}[theorem]{Corollary}
\theoremstyle{remark}

\numberwithin{equation}{section}
\allowdisplaybreaks[4]

\begin{document}

\title[Parabolic frequency on Ricci flow and Ricci-harmonic flow]{Parabolic frequency monotonicity on Ricci flow and Ricci-harmonic flow with bounded curvatures}


\author{Chuanhuan Li}
\address{School of Mathematics, Southeast University, Nanjing 211189, China}
\curraddr{}
\email{chli@seu.edu.cn}
\thanks{}

\author{Yi Li}
\address{School of Mathematics and Shing-Tung Yau Center of Southeast University, Southeast University, Nanjing 211189, China}
\curraddr{}
\email{yilicms@gmail.com, yilicms@seu.edu.cn}
\thanks{}

\author{Kairui Xu}
\address{School of Mathematics, Southeast University, Nanjing 211189, China}
\curraddr{}
\email{karry\_xu@seu.edu.cn}
\thanks{}


\date{}


\begin{abstract}
In this paper, we study the monotonicity of parabolic frequency motivated by \cite{frequency on RF} under the Ricci flow and the Ricci-harmonic flow on manifolds. Here we consider two cases: one is the  monotonicity of parabolic frequency for the solution of linear heat equation with bounded Bakry-\'{E}mery Ricci curvature, and another case is the  monotonicity of parabolic frequency for the solution of heat equation with bounded Ricci curvature.
\end{abstract}

\maketitle


\section{Introduction}

In 1979, the (elliptic) frequency functional for a harmonic function $u(x)$ on $\mathbb{R}^{n}$ was introduced by Almgren in \cite{Dirichlet problem by Almgren}, which is defined by
$$\displaystyle N(r)=\frac{r\int_{B(r,p)}|\nabla u(x)|^{2}dx}{\int_{\partial B(r,p)}u^{2}(x)dA},$$
where $dA$ is the induced $(n-1)$-dimensional Hausdorff measure on $\partial B(r,p)$ and $p$ is a fixed point in $\mathbb{R}^{n}$. Almgren obtained that $N(r)$ is monotone nondecreasing for $r$, and he used this property to investigate local regularity harmonic of functions and minimal surfaces. Next, Garofalo-Lin \cite{Mon prop by Garofalo,Unique continuation by Garofalo} considered the monotonicity of frequency functional on Riemannian manifold to study the unique continuation for elliptic operators. The frequency functional was also used to estimate the size of nodal sets in \cite{LA-2018,LA-2018 Yau}. For more applications, see \cite{Harmonic functions by Colding,H-H-L 1998,H-L 1994,frequency by Lin,Zelditch}.

The parabolic frequency for the solution of heat equation on $\mathbb{R}^{n}$ was introduced by Poon in \cite{Unique continuation by Poon}, and  Ni \cite{N-2015} considered the case when $u(t)$ is a holomorphic function, both of them showed that the parabolic frequency is nondecreasing. Besides, on Riemannian manifolds, the monotonicity of the parabolic frequency was obtained by Colding-Minicozzi \cite{frequency on manifold} through the drift Laplacian operator. Using the matrix Harnack's inequality in \cite{heat equation by Hamilton}, Li-Wang \cite{Parabolic frequency by Li-Wang} investigated the parabolic frequency on compact Riemannian manifolds and the 2-dimensional Ricci flow.
For higher dimension, the Ricci flow
\begin{align}
\partial_{t}g(t)=-2\text{\rm Ric}(g(t))
\label{RF}
\end{align}
is introduced by Hamilton in \cite{RF} to study the compact three-manifolds with positive Ricci curvature, which is a special case of the Poincar\'e conjecture finally proved by Perelman in \cite{Perelman 03,Perelman 02}. Hamilton \cite{RF} obtained the short time existence and uniqueness of the Ricci flow on compact manifolds, and  Shi \cite{Shi-89} obtained a short time solution of the Ricci flow on a complete noncompact manifold whose uniqueness with bounded Riemann curvature was proved by Chen-Zhu in \cite{C-Z 2006}.

In \cite{frequency on RF}, Julius-Dain  defined the following parabolic frequency for a solution $u(t)$ of the heat equation 
$$U(t)=-\frac{\tau\|\nabla_{g(t)}u(t)\|_{L^{2}(d\nu)}^{2}}{\|u(t)\|_{L^{2}(d\nu)}^{2}}e^{-\int\frac{1-\kappa}{\tau}}$$
where $\tau(t)$ is the backwards time, $\kappa(t)$ is the time-dependent function and $d\nu$ is the weighted measure. And they proved that parabolic frequency $U(t)$ for the solution of heat equation is monotone increasing along the Ricci
flow with the bounded Bakry-\'{E}mery Ricci curvature.


In this paper, besides considering the Ricci flow, we also consider the Ricci-harmonic flow
\begin{equation}
\left\{
\begin{aligned}
\partial_{t}g(t)&=-2\text{\rm Ric}(g(t))+2\alpha(t)d\phi(t)\otimes d\phi(t)\\
\partial_{t}\phi(t)&=\triangle_{g(t)}\phi(t)
\end{aligned}
\right.
\label{RHF}
\end{equation}
introduced in \cite{RHF by List,RHF by Reto}, which was motivated by Einstein vacuum equations in general relativity, where $g(t)$ is a family of Riemannian metrics, $\phi(t)$ is a family of smooth functions, and $\alpha(t)$ is a time-dependent positive function. M\"{u}ller \cite{RHF by Reto} got the short time existence and uniqueness of the Ricci-harmonic flow.  

The purpose of this paper is to study the monotonicity of parabolic frequency on Ricci flow $\eqref{RF}$ and the Ricci-harmonic flow $\eqref{RHF}$ on $[0,T)$ with $T\in(0,+\infty]$. Here, we consider the following two cases.

${}$

\textbf{A. Parabolic frequency for linear heat equation.}
We consider the following linear heat equation
\begin{equation}
(\partial_{t}-\triangle_{g(t)})u(t)=a(t)u(t),
\label{lpe}
\end{equation}
where $a(t)$ is a time-dependent smooth function, $u(t)$ is a family of smooth functions on $M$. Here, $g(t)$ is evolved by the Ricci flow or the Ricci-harmonic flow. 

The parabolic frequency $U(t)$ for the solution of linear heat equation $\eqref{lpe}$ on Ricci flow or  Ricci-harmonic flow is denoted by
\begin{equation}
U(t)=\exp\left\{-\int_{t_{0}}^{t}\frac{h'(s)+\kappa(s)}{h(s)} ds\right\}\frac{h(t)\int_{M}|\nabla_{g(t)}u(t)|_{g(t)}^{2}d V_{g(t)}}{\int_{M} u^{2}(t) d V_{g(t)}},
\end{equation}
where $h(t)$ and $\kappa(t)$ are  time-dependent functions, $[t_{0},t]\subset (0,T)$, and $dV_{g(t)}$ is defined in (2.4).

By calculating, on the Ricci flow, we obtain the monotonicity of the parabolic frequency for the solution of linear heat equation $\eqref{lpe}$.
\begin{theorem}
Suppose $(M^{n},g(t))_{t\in[0,T)}$ is the Ricci flow $\eqref{RF}$ with $\text{\rm Ric}_{f(t)}\leq \frac{\kappa(t)}{2h(t)}g(t)$, where ${\rm Ric}_{f(t)}$ is defined in  $\eqref{Baker}$.

$(i)$. If $h(t)$ is a negative time-dependent function, then the parabolic frequency $U(t)$ is monotone increasing along the Ricci flow.

$(ii)$. If $h(t)$ is a positive time-dependent function, then the parabolic frequency $U(t)$ is monotone decreasing along the Ricci flow.
\end{theorem}

Similarly, for the Ricci-harmonic flow, we have the following:
\begin{theorem}
Suppose $(M^{n},g(t),\phi(t))_{t\in[0,T)}$ is the Ricci-harmonic flow $\eqref{RHF}$ with $\text{\rm Ric}_{f(t)}-\alpha(t)\langle\nabla_{g(t)}\phi(t),\cdot\rangle_{g(t)}^{2}\leq \frac{\kappa(t)}{2h(t)}g(t)$.

(i). If $h(t)$ is a negative time-dependent function, then the parabolic frequency $U(t)$ is monotone increasing along the Ricci-harmonic flow.

(ii). If $h(t)$ is a positive time-dependent function, then the parabolic frequency $U(t)$ is monotone decreasing along the Ricci-harmonic flow.
\end{theorem}

${}$

\textbf{B. Parabolic frequency for heat equation.}
As also in Li-Wang \cite{Parabolic frequency by Li-Wang}, we consider the following heat equation:
\begin{equation}
\partial_{ t} u(t)=\triangle_{g(t)}u(t),
\label{heat equation}
\end{equation}
where $u(t)$ is a family of smooth functions on $M$.

We now investigate the monotonicity of the parabolic frequency  for the solutions of heat equation $\eqref{heat equation}$  on the Ricci flow $\eqref{RF}$ 
or the Ricci-harmonic flow $\eqref{RHF}$ 
with bounded ${\rm Ric}(g(t))$ instead of ${\rm Ric}_{f(t)}$. Note that, here we assume $M$ is closed.

The parabolic frequency $U(t)$ on $[t_{0},t_{1}]\subset (0,T)$ is denoted by
$$U(t)=\exp\left\{-\int_{t_{0}}^{t}\left(\frac{h'(s)}{h(s)}+2Kn+\frac{C(s)}{2}n+\frac{n}{s}\right)ds\right\}\frac{h(t)\int_{M}|\nabla_{g(t)}u(t)|_{g(t)}^{2}d V_{g(t)}}{\int_{M} u^{2}(t) d V_{g(t)}}$$
where $h(t)$ is a time-dependent function, $K$ is a positive constant, $n$ is the dimension of $M$, $C(t)=N/t$, $N=\log\frac{A}{\eta},\ \eta=\min_{M}u(0)$ and $A=\max_{ M}u(0)$. Observe that, $A$ and $\eta$ are both positive constants.

Then using Li-Yau Harnack estimate  and Hamilton's estimate on Ricci flow $\eqref{RF}$, we have the following:

\begin{theorem}
If $M^{n}$ is a closed Riemannian manifold, $(M,g(t))_{t\in[0,T)}$ is the solution of the Ricci flow $\eqref{RF}$ with bounded Ricci curvature, $0\leq\text{\rm Ric}(g(t))\leq Kg(t)$, where $K$ is a positive constant, and $u(t)$ is a positive solution of heat equation $\eqref{heat equation}$ with $u(\cdot,0)\leq A$, then the following holds.

$(i)$. If $h(t)$ is a negative time-dependent function, the parabolic frequency $U(t)$ is monotone increasing along the Ricci flow.

$(ii)$. If $h(t)$ is a positive time-dependent function, the parabolic frequency $U(t)$ is monotone decreasing along the Ricci flow.
\end{theorem}

With the same discussion of the Ricci-harmonic flow $\eqref{RHF}$, we have
\begin{theorem}
Suppose $M^{n}$ is a closed Riemannian manifold, $(M,g(t),\phi(t))_{t\in[0,T)}$ is the solution of the Ricci-harmonic flow $\eqref{RHF}$ with bounded Ricci curvature, $0\leq\text{\rm Ric}(g(t))\leq Kg(t)$, $K$ is a positive constant, and $u(t)$ is a positive solution of heat equation $\eqref{heat equation}$ with $u(\cdot,0)\leq A$. Moreover, if we assume that $\alpha(t)$ is a non-increasing function, bounded from below by $\bar{\alpha}$, and $0\leq d\phi(t)\otimes d\phi(t)\leq\frac{C}{t}g(t)$, where C is a constant depending on n and $\bar{\alpha}$, then the following holds.

$(i)$. If $h(t)$ is a negative time-dependent function, then the parabolic frequency $U(t)$ is monotone increasing along the Ricci-harmonic flow.

$(ii)$. If $h(t)$ is a positive time-dependent function, then the parabolic frequency $U(t)$ is monotone decreasing along the Ricci-harmonic flow.
\end{theorem}

\section{Notations and definitions }

In this section, we introduce some notations and definitions which will be used in the sequel. We use the notations in Hamilton’s paper \cite{RF} , $\nabla_{g}$ is the Levi-Civita connection induced by $g$, $\text{\rm Ric}$, R, $d\mu_{g}$ are  Ricci curvature, scalar curvature and volume form, respectively. The Laplacian of smooth time-dependent function $f(t)$ with respect to a family of Riemannian metrics $g(t)$ is
$$\triangle_{g(t)}f(t)=g^{ij}(t)[\partial_{i}\partial_{j}f(t)-\Gamma_{ij}^{k}(t)\partial_{k}f(t)]$$
where $\Gamma_{ij}^{k}(t)$ is Christoffel symbols of $g(t)$ and $\partial_{i}=\frac{\partial}{\partial x^{i}}$.


Under the Ricci flow $\eqref{RF}$, let  $\tau(t)=T-t$ be the backwards time. For any time-dependent function $f(t)\in C^{\infty}(M)$, we denote 
$$K(t)=(4\pi\tau(t))^{-\frac{n}{2}}e^{-f(t)}$$
to be the positive solution of the conjugate heat equation
\begin{equation}
\partial_{t}K(t)=-\triangle_{g(t)} K(t)+R(g(t)) K(t).
\end{equation}
From the definition of $K(t)$, we can prove the smooth function $f(t)$ satisfies the following equation
\begin{equation}
\partial_{t}f(t)=-\triangle_{g(t)} f(t)-R(g(t))+|\nabla_{g(t)} f(t)|_{g(t)}^{2}+\frac{n}{2\tau(t)} .
\end{equation}

With the same discussion, the backward heat-type equation with the conjugate heat equation under the Ricci-harmonic flow $\eqref{RHF}$ is given by
\begin{align}
\partial_{t}f(t)=-\triangle_{g(t)}f(t)-R(g(t)) +|\nabla_{g(t)}f(t)|_{g(t)}^{2}+\frac{n}{2\tau(t)}+\alpha(t)|\nabla_{g(t)}\phi(t)|_{g(t)}^{2}.
\end{align}
For (2.3), we can also define $K(t)$ as
$$K(t)=(4\pi\tau(t))^{-\frac{n}{2}}e^{-f(t)}$$
which satisfies the following conjugate heat equation
\begin{equation}
\partial_{t}K(t)=-\triangle_{g(t)} K(t)+R(g(t)) K(t)-\alpha(t)|\nabla_{g(t)}\phi(t)|_{g(t)}^{2}K(t).
\end{equation}
Then, under Ricci flow $\eqref{RF}$ or Ricci-harmonic flow $\eqref{RHF}$, we can define the weighted measure 
\begin{equation}
dV_{g(t)}:= K(t) d\mu_{g(t)}=(4\pi\tau(t))^{-\frac{n}{2}}e^{-f(t)}d\mu_{g(t)},\ \ \int_{M} dV_{g(t)}=1.
\end{equation}

On the weighted Riemannian manifold $(M^{n},g(t),dV_{g(t)})$, the weighted Bochner formula for $u$ is as follow 
\begin{align}
\label{Bochner}
\triangle_{g(t),f(t)}(|\nabla_{g(t)} u|_{g(t)}^{2})=&2|\nabla_{g(t)}^{2} u|_{g(t)}^{2}+2\langle\nabla_{g(t)} u,\nabla_{g(t)}\triangle_{g(t),f(t)}u\rangle_{g(t)}\\
&+2\text{\rm Ric}_{f(t)}(\nabla_{g(t)} u,\nabla_{g(t)} u)\notag
\end{align}
where
\begin{align}
\text{\rm Ric}_{f(t)}:=\text{\rm Ric}(g(t))+\nabla_{g(t)}^{2} f(t)
\label{Baker}
\end{align}
is the Bakry-\'{E}mery Ricci tensor introduced in \cite{B-E Ricci}, and
\begin{equation}
\triangle_{g(t),f(t)}u:=e^{f(t)}\text{div}_{g(t)}(e^{-f(t)}\nabla_{g(t)}{u})=\triangle_{g(t)} u-\langle\nabla_{g(t)} f(t),\nabla_{g(t)} u\rangle_{g(t)}
\end{equation}
is the drift Laplacian operator for any smooth function $u$.

Under the Ricci flow $\eqref{RF}$, the volume form $d\mu_{g(t)}$ satisfies
$$\partial_{t}(d\mu_{g(t)})=-R(g(t))d\mu_{g(t)}$$
while under the Ricci-harmonic flow $\eqref{RHF}$, the volume form $d\mu_{g}$ satisfies
$$\partial_{t}(d\mu_{g(t)})=\left(-R(g(t))+\alpha(t)|\nabla_{g(t)}\phi(t)|^{2}_{g(t)}\right)d\mu_{g(t)}.$$
Thus, the conjugate heat kernel measure $dV_{g(t)}$ are both evolved by
\begin{equation}
\partial_{t}(dV_{g(t)})=-(\triangle_{g(t)} K(t))d\mu_{g(t)}=-\frac{\triangle_{g(t)} K(t)}{K(t)}dV_{g(t)}.
\label{volume}
\end{equation}

For convenient, we use $\triangle_{f},\ \text{\rm Ric},\ \nabla,\ |\cdot|,\ dV$ to replace $\triangle_{g(t),f(t)},\ \text{\rm Ric}(g(t)),$ $\nabla_{g(t)},\ |\cdot|_{g(t)},\ dV_{g(t)}$. We always omit
the time variable t.

\section{parabolic frequency for the linear heat equation}

In this section, we consider parabolic frequency $U(t)$ for the solution of the linear heat equation $\eqref{lpe}$ under the Ricci flow $\eqref{RF}$ and the Ricci-harmonic flow $\eqref{RHF}$, respectively.

For a time-dependent function $u=u(t):M\times[t_{0},t_{1}]\rightarrow \mathbb{R}$ with $u(t), \partial_{t}u(t)\in W^{2,2}_{0}(dV_{g(t)})$ for all $t\in[t_{0},t_{1}]\subset(0,T)$, we denote by
\begin{align}
I(t)&=\int_{M} u^{2}(t) d V_{g(t)},\\
D(t)&=h(t)\int_{M}|\nabla_{g(t)}u(t)|_{g(t)}^{2}d V_{g(t)}\\
&=-h(t)\int_{M}\langle u(t),\triangle_{g(t),f(t)}u(t)\rangle_{g(t)} dV_{g(t)},\notag\\
U(t)&=\exp\left\{-\int_{t_{0}}^{t}\frac{h'(s)+\kappa(s)}{h(s)} ds\right\}\frac{D(t)}{I(t)},
\end{align}
where $h(t)$ and $\kappa(t)$ are both time-dependent smooth functions.

${}$

\subsection{Parabolic frequency for the linear heat equation under Ricci flow }
At first, we study the parabolic frequency $U(t)$ under the Ricci flow.
Before doing it, we give some lemmas.
\begin{lemma}
For all $u,v\in W^{1,2}_{0}(dV_{g(t)})$, the drift Laplacian operator $\triangle_{g(t),f(t)}$ satisfies:
$$\int_{M}(\triangle_{g(t),f(t)}u)vdV_{g(t)}=\int_{M}u(\triangle_{g(t),f(t)}v)dV_{g(t)}$$
\end{lemma}
\begin{proof}
By a straight computation, we obtain
\begin{align}
\int_{M}(\triangle_{g(t),f(t)}u)vdV_{g(t)}&=-(4\pi\tau(t))^{\frac{n}{2}}\int_{M}e^{-f(t)}\langle\nabla_{g(t)} u,\nabla_{g(t)} v\rangle_{g(t)} d\mu_{g(t)}\\
&=\int_{M}u(\triangle_{g(t),f(t)}v)dV_{g(t)}\notag
\end{align}
Thus, we get the desired result.
\end{proof}

\begin{lemma}
For any $u\in W^{2,2}_{0}(dV_{g(t)})$ , we have
\begin{align}
\int_{M}|\nabla_{g(t)}^{2} u|^{2}_{g(t)}d V_{g(t)}=\int_{M}\left(|\triangle_{g(t),f(t)}u|_{g(t)}^{2}-\text{\rm Ric}_{f(t)}(\nabla_{g(t)}u,\nabla_{g(t)}u)\right)d V_{g(t)}\notag
\end{align}
\end{lemma}
\begin{proof}
This result has been proved in Lemma 1.13 of \cite{frequency on RF}.
\end{proof}

\begin{lemma}
Suppose $u(t)$ is a family of smooth functions. Under the Ricci flow $\eqref{RF}$, the norm of $|\nabla_{g(t)} u(t)|_{g(t)}^{2}$ satisfies the evolution equation:
$$(\partial_{t}-\triangle_{g(t)})|\nabla_{g(t)} u(t)|_{g(t)}^{2}=-2|\nabla_{g(t)}^{2}u(t)|_{g(t)}^{2}+2\langle\nabla_{g(t)} u(t),\nabla_{g(t)}(\partial_{t}-\triangle_{g(t)}) u(t)\rangle_{g(t)}$$
\end{lemma}
\begin{proof}
By calculating straightly, we have
\begin{align}
\partial_{t}|\nabla u|^{2}=2\text{\rm Ric}(\nabla u,\nabla u)+2\langle\nabla u,\nabla\partial_{t}u\rangle.
\end{align}
Together with the Bochner formula, yields
\begin{align}
(\partial_{t}-\triangle)|\nabla u|^{2}=
2\langle\nabla u,\nabla(\partial_{t}-\triangle) u\rangle-2|\nabla^{2}u|^{2}.
\end{align}
Then we get the desired result.
\end{proof}

\begin{theorem}
Suppose $(M^{n},g(t))_{t\in[0,T)}$ is the Ricci flow $\eqref{RF}$ with $\text{\rm Ric}_{f(t)}\leq \frac{\kappa(t)}{2h(t)}g(t)$.

$(i)$. If $h(t)$ is a negative time-dependent function, then the parabolic frequency $U(t)$ is monotone increasing along the Ricci flow.

$(ii)$. If $h(t)$ is a positive time-dependent function, then the parabolic frequency $U(t)$ is monotone decreasing along the Ricci flow.
\end{theorem}
\begin{proof}
We only give the proof of $(i)$.
Our main purpose is to compute the $I'(t)$ and $D'(t)$. Under the Ricci flow $\eqref{RF}$,  combining with the linear heat equation $\eqref{lpe}$ and Lemma 3.3, we can obtain
\begin{align}
\label{3.6}
(\partial_{t}-\triangle)|\nabla u|^{2} 
&=2\langle\nabla u,\nabla(\partial_{t}-\triangle)u\rangle-2|\nabla^{2} u|^{2}\\
&=2a(t)|\nabla u|^{2}-2|\nabla^{2} u|^{2}.\notag
\end{align}

According to $\eqref{volume}$ and integration by parts, we get the derivative of $I(t)$ as follow
\begin{align}
I^{\prime}(t) &=\int_{M}\left(2u\partial_{t}u-u^{2}\frac{\triangle K}{K}\right)dV\\
&=\int_{M}\left(2u\partial_{t}u-\triangle (u^{2})\right)dV\notag\\
&=\int_{M}\left(2u\partial_{t}u-2u\triangle u-2|\nabla u|^{2}\right)dV\notag\\
&=-\frac{2}{h}D(t)+2a(t) I(t).\notag
\end{align}

If we write
$$\hat{I}(t)=\exp\left\{-\int_{t_{0}}^{t} 2a(s)ds\right\}I(t),$$
then we can easily find 
\begin{align}
\hat{I}^{\prime}(t)=-\frac{2}{h}\exp\left\{-\int_{t_{0}}^{t} 2a(s)ds\right\}D(t).
\label{derivative of I}
\end{align}

Next it's turn to compute the derivative of $D(t)$. Using $\eqref{volume}$, $\eqref{3.6}$ and the boundedness of Bakry-\'{E}mery Ricci tensor, we obtain
\begin{align}
D^{\prime}(t) &=h^{\prime}\int_{M}|\nabla u|^{2}dV+h\int_{M}\left(\partial_{t}|\nabla u|^{2}-|\nabla u|^{2}\frac{\triangle K}{K}\right)dV\\
&=h^{\prime}\int_{M}|\nabla u|^{2}dV+h\int_{M}(\partial_{t}-\triangle)|\nabla u|^{2}dV\notag\\
&=(2ah+h^{\prime})\int_{M}|\nabla u|^{2}dV-2h\int_{M}|\nabla^{2} u|^{2}dV\notag\\
&=(2ah+h^{\prime})\int_{M}|\nabla u|^{2}dV-2h\int_{M}\left[|\triangle_{f}u|^{2}-\text{\rm Ric}_{f}(\nabla u,\nabla u)\right]dV\notag\\
&\geq(\kappa+2ah+h^{\prime})\int_{M}|\nabla u|^{2}dV-2h\int_{M}|\triangle_{f}u|^{2}dV\notag\\
&=\left(2a+\frac{h^{\prime}+\kappa}{h}\right)D(t)-2h\int_{M}|\triangle_{f}u|^{2}dV.\notag
\end{align}
Similarly, if we write 
$$\hat{D}(t)=\exp\left\{-\int_{t_{0}}^{t} \left[2a(s)+\frac{h^{\prime}(s)+\kappa(s)}{h(s)}\right]ds\right\} D(t),$$
then we can find 
\begin{align}
\hat{D}^{\prime}(t)\geq-2h\exp\left\{-\int_{t_{0}}^{t} \left[2a(s)+\frac{h^{\prime}(s)+\kappa(s)}{h(s)}\right]ds\right\}\int_{M}|\triangle_{f}u|^{2}dV.
\label{derivative of D}
\end{align}

Finally, the parabolic frequency $U(t)$ can be written as $U(t)=\frac{\hat{D}(t)}{\hat{I}(t)}.$ By $\eqref{derivative of I}$ and $\eqref{derivative of D}$, we can compute the derivative of $U(t)$
\begin{align}
\hat{I}^{2}(t)U'(t) &=\hat{D}^{\prime}(t)\hat{I}(t)-\hat{I}^{\prime}(t)\hat{D}(t)\\
&\geq-2h\exp\left\{\int_{t_{0}}^{t} \left[4a(s)+\frac{h^{\prime}(s)+\kappa(s)}{h(s)}\right]ds\right\}\notag\\
&\quad\cdot\left[\left(\int_{M}|\triangle_{f}u|^{2}dV\right)\cdot\left(\int_{M}|u|^{2}dV\right)-\left(\int_{M}|\nabla u|^{2}dV\right)^{2}\right]\notag\\
&\geq-2h\exp\left\{\int_{t_{0}}^{t} \left[4a(s)+\frac{h^{\prime}(s)+\kappa(s)}{h(s)}\right]ds\right\}\notag\\
&\quad\cdot\left[\left(\int_M\langle u(t),\triangle_{f}u\rangle dV\right)^{2}-\left(\int_{M}|\nabla u|^{2}dV\right)^{2} \right]\notag\\
&= 0\notag
\end{align}
the last inequality is directly obtained by the definition of $D(t)$ the Cauchy-Schwarz inequality. 
\end{proof}

Then we have the following corollary.
\begin{corollary}
Let $(M^{n},g(t))_{t\in[0,T)}$ be Ricci flow $\eqref{RF}$ with $\text{\rm Ric}_{f(t)}\leq \frac{\kappa(t)}{2h(t)}g(t)$, where $h(t)$ is negative. If $u(\cdot,t_{1})=0$, then $u(\cdot,t)\equiv 0$ for any $t\in[t_{0},t_{1}]\subset(0,T)$.
\end{corollary}
\begin{proof}
Recalling the definition of $U(t)$, we get
\begin{align}
\frac{d}{dt}\log(I(t))&=\frac{I'(t)}{I(t)}=-\frac{2D(t)}{h(t)I(t)}+2a(t)\\
&=-\frac{2}{h(t)}\exp\left\{\int_{t_{0}}^{t}\frac{h'(s)+\kappa(s)}{h(s)} ds\right\}U(t)+2a(t).\notag
\label{log I}
\end{align}

According to Theorem 3.4 and integrating (3.13) from $t'$ to $t_{1}$ for any $t'\in[t_{0},t_{1}]$,  yields
\begin{align}
\log I(t_{1})-\log I(t')&\geq -2\int_{t'}^{t_{1}}\exp\left\{\int_{t_{0}}^{t}\frac{h'(s)+\kappa(s)}{h(s)} ds\right\}\frac{U(t)}{h(t)}dt+2\int_{t'}^{t_{1}}a(t)dt\notag\\
&\geq -2U(t_{0})\int_{t'}^{t_{1}}\exp\left\{\int_{t_{0}}^{t}\frac{h'(s)+\kappa(s)}{h(s)} ds\right\}\frac{dt}{h(t)}+2\int_{t'}^{t_{1}}a(t)dt\notag
\end{align}
Since $a(t),h(t)$  are finite, it follows from the last inequality that
\begin{align}
\frac{I(t_{1})}{I(t')}\geq\exp \left(-2U(t_{0})\int_{t'}^{t_{1}}\exp\left\{\int_{t_{0}}^{t}\frac{h'(s)+\kappa(s)}{h(s)} ds\right\}\frac{dt}{h(t)}+2\int_{t'}^{t_{1}}a(t)dt\right)
\end{align}
which implies Corollary 3.5.
\end{proof}

${}$

\subsection{Parabolic frequency for the linear heat equation under the Ricci-harmonic flow}
Next, we consider parabolic frequency  under the Ricci-harmonic flow $\eqref{RHF}$. Similarly, we will give the following Lemma.
\begin{lemma}
Under the Ricci-harmonic flow $\eqref{RHF}$, for a family of smooth functions $u(t)$ on $M$, the norm of $|\nabla_{g(t)} u(t)|_{g(t)}^{2}$ satisfies the evolution equation:
\begin{align}
(\partial_{t}-\triangle_{g(t)})|\nabla_{g(t)} u(t)|_{g(t)}^{2}=&2\langle\nabla_{g(t)} u(t),\nabla_{g(t)}(\partial_{t}-\triangle_{g(t)}) u(t)\rangle_{g(t)}\notag\\
&-2|\nabla_{g(t)}^{2} u(t)|_{g(t)}^{2}-2\alpha(t)\langle\nabla_{g(t)}\phi(t),\nabla_{g(t)} u(t)\rangle_{g(t)}^{2}\notag
\end{align}
\end{lemma}
\begin{proof}
By calculating straightly, we have
\begin{align}
\partial_{t}|\nabla u|^{2}=2\text{\rm Ric}(\nabla u,\nabla u)-2\alpha(t)\langle\nabla\phi,\nabla u\rangle^{2}+2\langle\nabla u,\nabla\partial_{t}u\rangle.
\end{align}
Together with the Bochner formula, yields
\begin{align}
(\partial_{t}-\triangle)|\nabla u|^{2}&=2\text{\rm Ric}(\nabla u,\nabla u)-2\alpha(t)\langle\nabla\phi,\nabla u\rangle^{2}+2\langle\nabla u,\nabla\partial_{t}u\rangle\\
&\quad-2|\nabla^{2} u|^{2}-2\langle\nabla u,\nabla\triangle u\rangle-2\text{\rm Ric}(\nabla u,\nabla u)\notag\\
&=-2|\nabla^{2} u|^{2}+2\langle\nabla u,\nabla(\partial_{t}-\triangle) u\rangle-2\alpha(t)\langle\nabla\phi,\nabla u\rangle^{2}\notag
\end{align}
Then we get the desired result.
\end{proof}
\begin{theorem}
Suppose $(M^{n},g(t),\phi(t))_{t\in[0,T)}$ is the Ricci-harmonic flow $\eqref{RHF}$ with $\text{\rm Ric}_{f(t)}-\alpha(t)\langle\nabla_{g(t)}\phi(t),\cdot\rangle_{g(t)}^{2}\leq \frac{\kappa(t)}{2h(t)}g(t)$,

(i) If $h(t)$ is a negative time-dependent function, then the parabolic frequency $U(t)$ is monotone increasing along the Ricci-harmonic flow.

(ii) If $h(t)$ is a positive time-dependent function, then the parabolic frequency $U(t)$ is monotone decreasing along the Ricci-harmonic flow.
\end{theorem}
\begin{proof}
With the same discussion in (3.8), using integration by parts, we have
\begin{align}
I'(t)
=-\frac{2}{h(t)}D(t)+2a(t)I(t)
\end{align}

According to the Bochner formula, together with Lemma 3.2 and Lemma 3.6, we have the derivative of $D(t)$
\begin{align}
D'(t)&=h'(t)\int_{M}|\nabla u|^{2} d V +h(t)\int_{M}\partial_{t}|\nabla u|^{2} d V+\int_{M}|\nabla u|^{2} \partial_{t}(d V)\\
&=h'(t)\int_{M}|\nabla u|^{2} d V+2h(t)\int_{M}(\partial_{t}-\triangle)|\nabla u|^{2}dV\notag\\
&=h'(t)\int_{M}|\nabla u|^{2} d V+2h(t)\int_{M}\left[\langle\nabla u,\nabla(\partial_{t}-\triangle)u\rangle-|\nabla^{2}u|^{2}\right.\notag\\
&\quad\left.-\alpha(t)\langle\nabla\phi,\nabla u\rangle^{2}\right]d V\notag\\
&=\left(\frac{h'(t)}{h(t)}+2a(t)\right)D(t)-2h(t)\int_{M}|\triangle_{f}u|^{2}d V\notag\\
&\quad +2h(t)\int_{M}\left[\text{\rm Ric}_{f}(\nabla u,\nabla u)-\alpha(t)\langle\nabla\phi,\nabla u\rangle^{2}\right]d V\notag
\end{align}

If $h(t)<0$, then according to (3.18), yields
\begin{align}
D'(t)\geq \left(\frac{h'(t)+\kappa(t)}{h(t)}+2a(t)\right)D(t)-2h(t)\int_{M}|\triangle_{f}u|^{2}d V
\end{align}

Thus, combining (3.17) and (3.19), applying H\"{o}lder's inequality, we have the following
\begin{align}
I^{2}(t)U'(t)&=\exp\left\{-\int_{t_{0}}^{t}\frac{h'(s)+\kappa(s)}{h(s)} ds\right\}\\ &\quad\cdot\left[-I(t)D(t)\left(\frac{h'(t)+\kappa(t)}{h(t)}\right)+I(t)D'(t)-I'(t)D(t)
\right]\notag\\
&\geq -2h(t)\exp\left\{-\int_{t_{0}}^{t}\frac{h'(s)+\kappa(s)}{h(s)} ds\right\}\notag\\
&\quad\cdot\left[\left(\int_{M}|\triangle_{f}u|^{2}d V\right)\left(\int_{M} u^{2} d V\right)-\left(\int_{M}|\nabla u|^{2}dV\right)^{2}\right]\notag\\
&\geq-2h(t)\exp\left\{-\int_{t_{0}}^{t}\frac{h'(s)+\kappa(s)}{h(s)} ds\right\}\notag\\
&\quad\cdot\left[\left(\int_M\langle u(t),\triangle_{f}u\rangle dV\right)^{2}-\left(\int_{M}|\nabla u|^{2}dV\right)^{2} \right]\notag\\
&= 0. \notag
\end{align}

Similarly, if $h(t)>0$ and recalling (3.18), we have 
\begin{align}
D'(t)\leq \left(\frac{h'(t)+\kappa(t)}{h(t)}+2a(t)\right)D(t)-2h(t)\int_{M}|\triangle_{f}u|^{2}d V.
\end{align}
Combining with (3.17) and (3.21), we obtain $I^{2}(t)U'(t)\leq 0$. Thus, we prove this Theorem.
\end{proof}

We define the first nonzero eigenvalue of the Ricci-harmonic flow $(M^{n},g(t),\phi(t))$ with the weighted measure $dV_{g(t)}$ by
$$\lambda(t)=\inf\left\{\left.\frac{\int_{M}|\nabla_{g(t)}u|^{2}_{g(t)}d V_{g(t)}}{\int_{M}u^{2}d V_{g(t)}}\right| u\in C^{\infty}(M)\setminus \{0\}\right\}.$$

Then we have the following corollary by Theorem 3.7.

\begin{corollary}
If $(M^{n},g(t),\phi(t))_{t\in[0,T)}$ is the Ricci-harmonic flow with $\text{\rm Ric}_{f(t)}-\alpha(t)\langle\nabla_{g(t)}\phi(t),\cdot\rangle_{g(t)}^{2}\leq \frac{\kappa(t)}{2h(t)}g(t)$, then for any $t\in[t_{0},t_{1}]\subset(0,T)$, the following holds.

(i) If $h(t)<0$, $h(t)\lambda(t)$ is a monotone increasing function.

(ii) If $h(t)>0$, $h(t)\lambda(t)$ is a monotone decreasing function.
\end{corollary}

If we take $\phi(t)$ as a time-dependent smooth function, we can get similar conclusion with Corollary 3.8 for Ricci flow.

\begin{corollary}
Suppose $(M^{n},g(t),\phi(t))_{t\in[0,T)}$ is the Ricci-harmonic flow $\eqref{RHF}$  with $\text{\rm Ric}_{f(t)}-\alpha(t)\langle\nabla_{g(t)}\phi(t),\cdot\rangle_{g(t)}^{2}\leq \frac{\kappa(t)}{2h(t)}g(t)$, where $h(t)$ is negative. If $h(t)<0$ and $u(\cdot,t_{1})=0$, then $u(\cdot,t)\equiv 0$ for any $t\in[t_{0},t_{1}]\in(0,T)$.
\end{corollary}
\begin{proof}
With the same discussion of Corollary 3.5, we obtain the desired result.
\end{proof}

\section{parabolic frequency for heat equation}

In this section, we study the parabolic frequency for the solution of heat equation $\eqref{heat equation}$ under the closed Ricci flow $\eqref{RF}$ and the closed Ricci-harmonic flow $\eqref{RHF}$, here we use Li-Yau Harnack inequality in \cite{heat equation under RF by Cao xiaodong} and Hamilton's estimate in \cite{heat equation by Hamilton} to replace bounded Bakry-\'{E}mery Ricci curvature by bounded Ricci curvature.

For a function $u=u(t):M\times[t_{0},t_{1}]\rightarrow \mathbb{R}^{+}$ with $u(t),\partial_{t}u(t)\in W_{0}^{2,2}(dV_{g(t)})$ and for all $t\in[t_{0},t_{1}]\subset(0,T)$, we define the parabolic frequency $U(t)$
\begin{align}
I(t)&=\int_{M} u^{2}(t) d V_{g(t)},\\
D(t)&=h(t)\int_{M}|\nabla_{g(t)}u(t)|_{g(t)}^{2}d V_{g(t)},\\
U(t)&=\exp\left\{-\int_{t_{0}}^{t}\left(\frac{h'(s)}{h(s)}+2Kn+\frac{C(s)}{2}n+\frac{n}{s}\right) ds\right\}\frac{D(t)}{I(t)},
\end{align}
where $h(t)$ is a time-dependent function, $K$ is a positive constant, $n$ is the dimension of $M$, $C(t)=N/t$, $N=\log\frac{A}{\eta},\ \eta=\min_{M}u(0)$ and $A=\max_{M}u(0)$. Observe that, $A$ and $\eta$ are both positive constants.

${}$

\subsection{Parabolic frequency for heat equation under closed Ricci flow}
Before computing the monotonicity of $U(t)$, we give some lemmas from \cite{heat equation under RF by Cao xiaodong} and \cite{heat equation by Hamilton}.
\begin{lemma}
If $M^{n}$ is a closed Riemannian manifold, $(M^{n},g(t))$ is the solution of the Ricci flow $\eqref{RF}$ and $u(t)$ is a positive solution of heat equation $\eqref{heat equation}$ with $u(\cdot,0)\leq A$, then we have the following estimate:
$$t|\nabla_{g(t)} u(t)|_{g(t)}^{2}\leq u^{2}(t)\log\left(\frac{A}{u(t)}\right).$$
\end{lemma}
\begin{proof}
This Lemma has been proved in Theorem 2.4 of \cite{heat equation under RF by Cao xiaodong} and Theorem 3.1(b) of \cite{gradient estimate under RF}.
\end{proof}

\begin{lemma}
If $M^{n}$ is a closed Riemannian manifold, $(M,g(t))$ is the solution of the Ricci flow $\eqref{RF}$ with bounded Ricci curvature, $0\leq\text{\rm Ric}(g(t))\leq Kg(t)$, $K$ is a positive constant, and $u(t)$ is a positive solution of heat equation $\eqref{heat equation}$, then we have the following estimate:
$$\frac{|\nabla_{g(t)} u(t)|_{g(t)}^{2}}{u(t)}-\partial_{t}u(t)\leq \frac{n}{2t}u(t)+Knu(t)$$
\end{lemma}
\begin{proof}
This lemma is from Theorem 2.9 in \cite{heat equation under RF by Cao xiaodong}, here we give another method from \cite{heat equation by Hamilton} to prove it.
Recalling Lemma 4.1, we have
\begin{equation}
\frac{\partial}{\partial t}\frac{|\nabla u|^{2}}{u}=\triangle\frac{|\nabla u|^{2}}{u}-\frac{2}{u}\left|\nabla_{i}\nabla_{j}u-\frac{\nabla_{i}u\nabla_{j}u}{u}\right|^{2},
\end{equation}
taking trace over the second term, yields
\begin{equation}
\left|\nabla_{i}\nabla_{j}u-\frac{\nabla_{i}u\nabla_{j}u}{u}\right|^{2}\geq\frac{1}{n}\left(\triangle u-\frac{|\nabla u|^{2}}{u}\right)^{2}=\frac{1}{n}\left(\partial_{ t}u-\frac{|\nabla u|^{2}}{u}\right)^{2}.
\end{equation}

Next, we need calculate the following
\begin{align}
\frac{\partial}{\partial t}\left(\partial_{ t} u\right)
=2R_{ij}\nabla_{i}\nabla_{j}u+\triangle\left(\partial_{ t}u\right)
\end{align}

If we let
$$\psi=\partial_{ t}u-\frac{|\nabla u|^{2}}{u}+\frac{n}{2t}u+Knu$$
together with (4.4)-(4.6), then
\begin{align}
\frac{\partial\psi}{\partial t}
\geq\triangle\psi+\frac{2}{nu}\left(\partial_{ t}u-\frac{|\nabla u|^{2}}{u}\right)^{2}+2R_{ij}\nabla_{i}\nabla_{j}u-\frac{n}{2t^{2}}u
\end{align}
Now, when $\psi\leq 0$, we have
$$0\leq Knu+\frac{n}{2t}u\leq \frac{|\nabla u|^{2}}{u}-\partial_{ t}u.$$
which implies
$$\frac{\partial\psi}{\partial t}\geq\triangle\psi\ \ \ \ \text{when}\ \ \ \psi\leq0.$$
Note that $\psi\rightarrow +\infty \ \ \text{as}\ \ t\rightarrow 0$, according to the maximum principle, we obtain $\psi\geq 0$ for all time. Then we proved this Lemma.
\end{proof}

\begin{theorem}
If $M^{n}$ is a closed Riemannian manifold, $(M,g(t))_{t\in[0,T)}$ is the solution of the Ricci flow $\eqref{RF}$ with bounded Ricci curvature, $0\leq\text{\rm Ric}(g(t))\leq Kg(t)$, where $K$ is a positive constant, and $u(t)$ is a positive solution of heat equation $\eqref{heat equation}$ with $u(\cdot,0)\leq A$, then the following holds.

$(i)$. If $h(t)$ is a negative time-dependent function, the parabolic frequency $U(t)$ is monotone increasing along the Ricci flow.

$(ii)$. If $h(t)$ is a positive time-dependent function, the parabolic frequency $U(t)$ is monotone decreasing along the Ricci flow.
\end{theorem}
\begin{proof}
Before discussing the monotonicity of $U(t)$, we need calculate the derivative of $I(t)$ and $D(t)$. Using Young's identity and Lemma 4.1, we get the derivative of $I(t)$.
\begin{align}
I'(t)&=\frac{d}{dt}\left(\int_{M}u^{2} dV\right)\\
&=2\int_{M}\left(u\cdot\partial_{t}u-|\nabla u|^{2}\right)d V -2\int_{M}u\triangle u dV\notag\\
&\geq-\left(\frac{n}{t}+2Kn\right)I(t)-2\int_{M}u\triangle u dV\notag\\
&\geq-\left(\frac{n}{t}+2Kn+\frac{C(t)}{2}n\right)I(t)-\frac{2}{C(t)n}\int_{M}|\triangle u|^{2}dV.\notag
\end{align}

For the derivative of $D(t)$, from Lemma 3.3, we have
\begin{align}
D'(t)&=h'(t)\int_{M}|\nabla u|^{2}d V+h(t)\frac{d}{dt}\left(\int_{M}|\nabla u|^{2}dV\right)\\
&=h'(t)\int_{M}|\nabla u|^{2}d V+h(t)\int_{M}(\partial_{t}-\triangle)|\nabla u|^{2}dV\notag\\
&=h'(t)\int_{M}|\nabla u|^{2}d V-2h(t)\int_{M}|\nabla^{2}u|^{2}dV.\notag
\end{align}
If $h(t)< 0$, then combing (4.8) and (4.9), together with Lemma 4.1, yields
\begin{align}
I^{2}(t)&U'(t)\geq \exp\left\{-\int_{t_{0}}^{t}\left(\frac{h'(s)}{h(s)}+2Kn+\frac{C(s)}{2}n+\frac{n}{s}\right) ds\right\}\notag\\
&\quad\cdot\left[-2h I(t)\left(\int_{M}|\nabla^{2}u|^{2}dV\right)+\frac{2h}{C(t)n}\left(\int_{M}|\triangle u|^{2}dV\right)\left(\int_{M}|\nabla u|^{2}dV\right)\right]\notag\\
&\geq \exp\left\{-\int_{t_{0}}^{t}\left(\frac{h'(s)}{h(s)}+2Kn+\frac{C(s)}{2}n+\frac{n}{s}\right) ds\right\}\notag\\
&\quad\cdot\left[-\frac{2h}{n}I(t)\left(\int_{M}|\triangle u|^{2}dV\right)+\frac{2h}{C(t)n}\left(\int_{M}|\triangle u|^{2}dV\right)\left(\int_{M}|\nabla u|^{2}dV\right)\right]\notag\\
&\geq \exp\left\{-\int_{t_{0}}^{t}\left(\frac{h'(s)}{h(s)}+2Kn+\frac{C(s)}{2}n+\frac{n}{s}\right) ds\right\}\notag\\
&\quad\cdot\left[-\frac{2h}{n}I(t)\left(\int_{M}|\triangle u|^{2}dV\right)+\frac{2h}{C(t)n}\cdot\frac{M}{t}I(t)\left(\int_{M}|\triangle u|^{2}dV\right)\right]\notag\\
&=0\notag
\end{align}
where we take trace over $|\nabla^{2}u|^{2}$ and let $C(t)=\frac{N}{t}$.

On the other hand, if $h(t)>0$, similarly, we have
\begin{align}
I^{2}(t)&U'(t)\leq \exp\left\{-\int_{t_{0}}^{t}\left(\frac{h'(s)}{h(s)}+2Kn+\frac{C(s)}{2}n+\frac{n}{s}\right) ds\right\}\notag\\
&\quad\cdot\left[-2hI(t)\left(\int_{M}|\nabla^{2}u|^{2}dV\right)+\frac{2h}{C(t)n}\left(\int_{M}|\triangle u|^{2}dV\right)\left(\int_{M}|\nabla u|^{2}dV\right)\right]\notag\\
&\leq \exp\left\{-\int_{t_{0}}^{t}\left(\frac{h'(s)}{h(s)}+2Kn+\frac{C(s)}{2}n+\frac{n}{s}\right) ds\right\}\notag\\
&\quad\cdot\left[-\frac{2h}{n}I(t)\left(\int_{M}|\triangle u|^{2}dV\right)+\frac{2h}{C(t)n}\cdot\frac{M}{t}I(t)\left(\int_{M}|\triangle u|^{2}dV\right)\right]\notag\\
&=0.\notag
\end{align}

Thus we get our result.
\end{proof}

We define the first nonzero eigenvalue of the Ricci flow $(M^{n},g(t))$ with the weighted measure $dV_{g(t)}$ by
$$\lambda(t)=\inf\left\{\left.\frac{\int_{M}|\nabla_{g(t)}u|_{g(t)}^{2}d V_{g(t)}}{\int_{M}u^{2}d V_{g(t)}}\right| 0<u\in C^{\infty}(M)\setminus \{0\}\right\}$$

Then we have the following corollary by Theorem 4.3.

\begin{corollary}
If $M^{n}$ is a closed Riemannian manifold, $(M,g(t))_{t\in[0,T)}$ is the solution of the Ricci flow $\eqref{RF}$ with bounded Ricci curvature, $0\leq\text{\rm Ric}(g(t))\leq Kg(t)$, where $K$ is a positive constant, and $u(t)$ is a positive solution of heat equation $\eqref{heat equation}$ with $u(0)\leq A$, then for any $t\in[t_{0},t_{1}]\subset(0,T)$, the following holds.

$(i)$. If $h(t)$ is a negative time-dependent function,, then $h(t)\lambda(t)$ is a monotone increasing function.

$(ii)$. If $h(t)$ is a positive time-dependent function,, then $h(t)\lambda(t)$ is a monotone decreasing function.
\end{corollary}

${}$

\subsection{Parabolic frequency for heat equation under closed Ricci-harmonic flow}
For Ricci-harmonic flow, we can prove that the conclusions of Lemma 4.1 and Lemma 4.2 still hold.
\begin{lemma}
Suppose $M^{n}$ is a closed Riemannian manifold, $(M^{n},g(t),\phi(t))$ is the solution of the Ricci-harmonic flow $\eqref{RHF}$, and u(t) is a positive solution of heat equation $\eqref{heat equation}$ with $u(\cdot,0)\leq A$. Moreover, if we assume $\alpha(t)\geq\bar{\alpha}>0$ and $d\phi\otimes d\phi\geq0$, then we have the following estimate:
$$t|\nabla_{g(t)} u(t)|_{g(t)}^{2}\leq u^{2}(t)\log\left({\frac{A}{u(t)}}\right)$$
\end{lemma}
\begin{proof}
The proof of this lemma is similar with Theorem 3.1(b) of \cite{gradient estimate under RF}.
According to the maximum principle and the heat equation $\eqref{heat equation}$, combing $u(\cdot,0)\leq A$, we have
$$\sup_{M\times [0,T]}u(x,t)\leq A.$$

Using the Bochner formula $\eqref{Bochner}$, we computer
\begin{equation}
\frac{\partial}{\partial t}\frac{|\nabla u|^{2}}{u}=\triangle\frac{|\nabla u|^{2}}{u}-\frac{2}{u}\left|\nabla_{i}\nabla_{j}u-\frac{\nabla_{i}u\nabla_{j}u}{u}\right|^{2}-\frac{2}{u}\alpha(t)\langle\nabla\phi,\nabla u\rangle^{2}
\end{equation}
and
\begin{equation}
\frac{\partial}{\partial t}\left(u\log\frac{A}{u}\right)=\triangle\left(u\log\frac{A}{u}\right)+\frac{|\nabla u|^{2}}{u}.
\end{equation}

Then we put
$$\psi=t\frac{|\nabla u|^{2}}{u}-u\log\frac{A}{u}.$$
Obviously, $\psi\leq 0$ at $t=0$, together with (4.10) and (4.11), throwing away $\alpha(t)\langle\nabla\phi,\nabla u\rangle^{2}$ with $d\phi\otimes d\phi\geq0$ and $\alpha(t)\geq\bar{\alpha}>0$, yields
\begin{align}
\frac{\partial}{\partial t}\psi
\leq\triangle\psi.\notag
\end{align}
Recalling the maximum principle, we have $\phi\leq 0$ for all time, which means
$$t|\nabla u|^{2}\leq u^{2}\log(\frac{A}{u}).$$
Thus we obtain this lemma.
\end{proof}

\begin{lemma}
Suppose $M^{n}$ is a compact Riemannian manifold, $(M,g(t),\phi(t))$ is the solution of the Ricci-harmonic flow $\eqref{RHF}$ with bounded Ricci curvature, $0\leq\text{\rm Ric}(g(t))\leq Kg(t)$, $K$ is a positive constant, and $u(t)$ is a positive solution of heat equation $\eqref{heat equation}$. Moreover, if we assume $\alpha(t)$ is a non-increasing function, bounded from below by $\bar{\alpha}$, and $0\leq d\phi\otimes d\phi\leq\frac{C}{t}g(t)$, where C is a constant depending on n and $\bar{\alpha}$, then we have the following estimate:
$$\frac{|\nabla_{g(t)} u(t)|_{g(t)}^{2}}{u(t)}-\partial_{t} u(t)\leq \frac{C_{n}}{2t}u(t)+Knu(t)$$
where $C_{n}=\frac{n}{2}+4nC\alpha(0)$.
\end{lemma}
\begin{proof}
For more details, see \cite{heat equation under RHF}.
\end{proof}


\begin{theorem}
Suppose $M^{n}$ is a closed Riemannian manifold, $(M,g(t),\phi(t))_{t\in[0,T)}$ is the solution of the Ricci-harmonic flow $\eqref{RHF}$ with bounded Ricci curvature, $0\leq\text{\rm Ric}(g(t))\leq Kg(t)$, $K$ is a positive constant, and $u(t)$ is a positive solution of heat equation $\eqref{heat equation}$ with $u(\cdot,0)\leq A$. Moreover, if we assume that $\alpha(t)$ is a non-increasing function, bounded from below by $\bar{\alpha}$, and $0\leq d\phi(t)\otimes d\phi(t)\leq\frac{C}{t}g(t)$, where C is a constant depending on n and $\bar{\alpha}$, then the following holds.

$(i)$. If $h(t)$ is a negative time-dependent function, then the parabolic frequency $U(t)$ is monotone increasing along the Ricci-harmonic flow.

$(ii)$. If $h(t)$ is a positive time-dependent function, then the parabolic frequency $U(t)$ is monotone decreasing along the Ricci-harmonic flow.
\end{theorem}
\begin{proof}
The estimate of $I'(t)$ is same as the proof of Theorem 4.3. That is
$$I'(t)\geq-\left(\frac{C_{n}}{t}+2Kn+\frac{C(t)}{2}n\right)I(t)-\frac{2}{C(t)n}\int_{M}|\triangle u|^{2}dV.$$
If we write 
$$\hat{I}(t)=\exp\left\{\int_{t_{0}}^{t}\left(\frac{C_{n}}{s}+2Kn+\frac{C(s)}{2}n\right)ds\right\}I(t)$$
then we can easily find 
$$\hat{I}'(t)\geq-\frac{2}{C(t)n}\exp\left\{\int_{t_{0}}^{t}\left(\frac{C_{n}}{s}+2Kn+\frac{C(s)}{2}n\right)ds\right\}\int_{M}|\triangle u|^{2}dV.$$
For the computation of $D'(t)$, according to Lemma 3.6, we have
\begin{align}
    D'(t)&=h'(t)\int_{M}|\nabla u|^{2}dV+h(t)\int_{M}(\partial_{t}-\triangle)|\nabla u|^{2}dV \notag\\
    &=h'(t)\int_{M}|\nabla u|^{2}dV-2h(t)\int_{M}\left(\alpha(t)d\phi\otimes d\phi(\nabla u,\nabla u)+|\nabla^{2}u|^{2}\right)dV \notag\\
    &=\frac{h'(t)}{h(t)}D(t)-2h(t)\int_{M}\left(\alpha(t)d\phi\otimes d\phi(\nabla u,\nabla u)+|\nabla^{2}u|^{2}\right) dV. \notag
\end{align}
Similarly, if we write $$\hat{D}(t)=\exp\left\{-\int_{t_{0}}^{t}\frac{h'(s)}{h(s)}ds\right\}D(t)$$
then we will find
$$\hat{D'}(t)=-2h(t)\exp\left\{-\int_{t_{0}}^{t}\frac{h'(s)}{h(s)}ds\right\}\int_{M}\left(\alpha(t)d\phi\otimes d\phi(\nabla u,\nabla u)+|\nabla^{2}u|^{2}\right)dV.$$
Thus, the parabolic frequency can be defined by $U(t)=\frac{\hat{D}(t)}{\hat{I}(t)}$.
Combining the boundedness of $\alpha$ and $\phi$, if $h(t)$ is a negative time-dependent function, we see that 
\begin{align}
    &\hat{I}^{2}(t)U'(t)=\hat{D}'(t)\hat{I}(t)-\hat{D}(t)\hat{I}'(t) \\
    &\geq \exp\left\{\int_{t_{0}}^{t}\left(-\frac{h'(s)}{h(s)}+\frac{C_{n}}{s}+2Kn+\frac{C(s)}{2}n\right)ds\right\}\notag \\
    &\quad\cdot\left[-2h\left(\int_{M}u^{2}dV\right)\left(\int_{M}\left(\alpha(t)d\phi\otimes d\phi(\nabla u,\nabla u)+|\nabla^{2}u|^{2}\right)dV\right)\right.\notag \\
    &\quad\left.+\frac{2h}{C(t)n}\left(\int_{M}|\triangle u|^{2}dV\right)\left(\int_{M}|\nabla u|^{2}dV\right)\right] \notag\\
    &\geq \exp\left\{\int_{t_{0}}^{t}\left(-\frac{h'(s)}{h(s)}+\frac{C_{n}}{s}+2Kn+\frac{C(s)}{2}n\right)ds\right\}\notag\\
    &\quad\cdot\left[-\frac{2h}{n}\left(\int_{M}u^{2}dV\right)\left(\int_{M}|\triangle u|^{2}dV\right)+\frac{2h}{C(t)n}\cdot\frac{N}{t}\left(\int_{M}|\triangle u|^{2}dV\right)\left(\int_{M} u^{2}dV\right)\right]\notag\\
    &= 0\notag
\end{align}
where we take $C(t)=\frac{N}{t}$.

With the same discussion, if $h(t)$ is a positive time-dependent function, the parabolic frequency $U'(t)\leq 0$. Thus, we get the desired results.
\end{proof}

\bibliographystyle{amsplain}

\end{document}